\documentclass[12pt,reqno]{amsart}

\usepackage{amsmath}
\usepackage{amsfonts}
\usepackage{amssymb}
\usepackage{amsthm}
\usepackage{mathrsfs}

\usepackage{parskip}

\usepackage{enumerate}
\usepackage{geometry}
\usepackage{hyperref}

\newcommand{\RR}{\mathbf{R}}

\newcommand{\ZZ}{\mathbf{Z}}

\DeclareMathOperator{\ind}{ind}

\DeclareMathOperator{\area}{area}

\DeclareMathOperator{\genus}{genus}

\theoremstyle{plain} \newtheorem{defi}{Definition}
\theoremstyle{plain} 
\theoremstyle{plain} \newtheorem{theo}[defi]{Theorem}
\theoremstyle{plain} 
\theoremstyle{plain} 
\theoremstyle{plain} \newtheorem{lemm}[defi]{Lemma}
\theoremstyle{plain} 
\theoremstyle{plain} 
\theoremstyle{plain} \newtheorem*{theo*}{Theorem}
\theoremstyle{plain} 
\theoremstyle{plain} 
\theoremstyle{definition} \newtheorem*{clai}{Claim}

\numberwithin{defi}{section} 
\numberwithin{equation}{section} 

\allowdisplaybreaks

\title{Minimal hypersurfaces with arbitrarily large area}
\author{Otis Chodosh}
\address{Department of Mathematics\\Princeton
University\\Princeton, NJ 08544}
\address{School of Mathematics\\Institute for Advanced Study\\Princeton, NJ 08540}
\email{ochodosh@math.princeton.edu}
\author{Christos Mantoulidis}
\address{Department of Mathematics\\ Massachusetts Institute of Technology \\Cambridge, MA 02139}
\email{c.mantoulidis@mit.edu}
\date{\today}
\begin{document}

\begin{abstract}
For $3\leq n\leq 7$, we prove that a bumpy closed Riemannian $n$-manifold contains a sequence of connected embedded closed minimal surfaces with unbounded area. 
\end{abstract}

\maketitle

\section{Introduction}

The goal of this note is to prove the following result:

\begin{theo}\label{theo:main}
For $3\leq n \leq 7$, suppose that $(M^{n},g)$ is a connected, closed Riemannian $n$-manifold with a bumpy metric. Then, there is a sequence $\Sigma_{j} \subset (M^{n},g)$ of connected, embedded, closed minimal surfaces with $\area(\Sigma_{j})\to\infty$. 
\end{theo}

We recall here the following standard definition:
\begin{defi} \label{def:bumpy.metric}
	We say that a metric $g$ on a Riemannian manifold $M^{n}$ is \emph{bumpy} if there is no immersed closed minimal hypersurface $\Sigma^{n-1}$ with a non-trivial Jacobi field. 
\end{defi}

By work of White \cite{White:bumpy.old,White:bumpy.new}, bumpy metrics are generic in the sense of Baire category. Here, ``generic'' will always mean in the Baire category sense. 

The key quantifier in Theorem \ref{theo:main} is \emph{connected}. Indeed, thanks to the resolution of Marques--Neves's multiplicity-one conjecture by the authors \cite{ChodoshMantoulidis} for $n=3$ and recently by Zhou \cite{Zhou:mult-one} for $3\leq n\leq 7$, if one does not require that the $\Sigma_{j}$ are connected, then Theorem \ref{theo:main} would be an immediate consequence of the following statement concerning the existence of minimal surfaces $\Sigma_{p}$ realizing the $p$-widths (note that if $(M^{n},g)$ does not satisfy the Frankel property, then it is not clear that the number of connected components of $\Sigma_{p}$ is uniformly bounded): 

\begin{theo}[$n=3$ \cite{ChodoshMantoulidis}, $3\leq n\leq 7$ \cite{Zhou:mult-one,MarquesNeves:uper-semi-index}]\label{theo:exist-p-widths}
For $3\leq n\leq 7$, suppose that $(M^{n},g)$ is a closed Riemannian $n$-manifold with a bumpy metric. Then, there is a constant $C=C(M,g)>0$ so that for each positive integer $p$, there is a smooth embedded closed minimal surface $\Sigma_{p}$ so that
	\begin{itemize}
		\item each component of $\Sigma_{p}$ is two-sided,
		\item the area of $\Sigma_{p}$ satisfies $C^{-1} p^{\frac 1 n}\leq \area_{g}(\Sigma_{p}) \leq C p^{\frac 1 n}$, and
		\item the index of $\Sigma_{p}$ is satisfies $\ind(\Sigma_{p}) = p$.
	\end{itemize}
\end{theo}

We emphasize that the $p$-widths were introduced by Gromov, Guth, and Marques--Neves \cite{Gromov:waist,Guth:minimax,MarquesNeves:posRic} in the Almgren--Pitts setting (as considered in the work of Zhou \cite{Zhou:mult-one}) and were understood in the Allen--Cahn setting (as considered in the work of the authors \cite{ChodoshMantoulidis}) by Gaspar--Guaraco \cite{Guaraco,GasparGuaraco}. 

Recently, there has been exciting progress on the existence and behavior of min-max minimal hypersurfaces in Riemannian manifolds; besides those works already mentioned, we point out that Liokumovich--Marques--Neves \cite{LMN:Weyl} have proven a Weyl law for the $p$-widths. This a key component of the proof by Irie--Marques--Neves that minimal surfaces are generically dense \cite{IrieMarquesNeves}. These results were extended to the Allen--Cahn setting by Gaspar--Guaraco \cite{GasparGuaraco:weyl}. Marques--Neves--Song have proven that generically there is an equidistributed set of minimal hypersurfaces \cite{MarquesNevesSong}. Song has proven Yau's conjecture on infinitely many minimal surfaces in all cases \cite{Song:Yau}. Finally, we note that Zhou--Zhu's min-max theory for prescribed mean curvature \cite{ZhouZhu:prescribed} was instrumental in the proof of Zhou's multiplicity-one result.

Theorem \ref{theo:main} is a consequence of the following statement:
\begin{theo} \label{theo:area-unbounded}
	For $3\leq n\leq 7$, suppose that $(M^{n},g)$ is a connected Riemannian $n$-manifold with a bumpy metric. Then, either:
	\begin{enumerate}
		\item  there exists a sequence of connected closed embedded stable minimal hypersurfaces $\Sigma_{j}\subset (M^{n},g)$ with $\area_{g}(\Sigma_{j})\to\infty$, or 
		\item the hypersurfaces $\Sigma_{p}$ from Theorem \ref{theo:exist-p-widths} have at least one connected component $\Sigma_{p}'$ with $\area_{g}(\Sigma_{p}') \geq C p^{\frac 1n}$ for some $C=C(M,g)>0$ independent of $p$.
	\end{enumerate}
\end{theo} 

We note that for $(M^{3},g)$ with a bumpy metric of positive scalar curvature, the first case never occurs by \cite{CKM,Carlotto:arb-large} and moreover we can conclude that $\ind(\Sigma_{p}')\to\infty$. On the other hand, Colding--Minicozzi have shown that any $3$-manifold admits a (bumpy) metric in which the first alternative occurs \cite{ColdingMinicozzi:no-area-bds}.

We further note that in \cite{ChodoshMantoulidis} we have shown how to use the monotonicity formula find a component $\Sigma_{p}''$ of $\Sigma_{p}$ with $\genus(\Sigma_{p}'')\geq C^{-1}\ind(\Sigma_{p}'')\geq p^{\frac 23}$ (the genus bound follows from the work of Ejiri--Micallef \cite{EjiriMicallef}). The same argument applies in higher dimensions to find $\Sigma_{p}''\subset (M^{n},g)$ with $\ind(\Sigma_{p}'') \geq Cp^{1-\frac 1n}$. It is not clear to us, however, that $\Sigma_{p}'$ and $\Sigma_{p}''$ can be taken to be the same. 

\subsection{Acknowledgments}

O.C. was supported in part by an NSF grant DMS-1811059. He is grateful to Simon Brendle and Davi Maximo for several interesting conversations concerning the behavior of the $\Sigma_{p}$ surfaces. C.M. would like to thank Yevgeny Liokumovich for interesting conversations on Almgren--Pitts theory. Both authors would like to thank Andr\'e Neves for a helpful discussion about this work. This work originated during the authors' visit to the International Centre for Mathematical Sciences (ICMS) in Edinburgh during the ``Geometric Analysis Workshop'' during the summer of 2018, which they would like to acknowledge for its support.

\section{Proof of Theorem \ref{theo:area-unbounded}}

	\begin{lemm}\label{lemm:stab-finitely-many}
	For $3\leq n\leq 7$, consider $(M^{n},g)$ a closed Riemannian $n$-manifold with a bumpy metric. Then for $A_{0}>0$, there are finitely many closed embedded stable minimal hypersurfaces with area at most $A_{0}$. 
	\end{lemm}
	\begin{proof}
		This follows from curvature estimates \cite{Sch83,SSY,Schoen-Simon:1981} and the proof of \cite{ColdingMinicozzi:no-area-bds} (see also the proof of \cite{Sharp}). Namely, if there is a sequence $\Sigma_{j}$ of stable embedded minimal hypersurfaces with uniformly bounded area, then they have uniformly bounded curvature. Thus, after passing to a subsequence, they converge smoothly (possibly with multiplicity) to a stable minimal hypersuface $\Sigma_{\infty}$. Passing to a double cover if $\Sigma_{\infty}$ is one-sided, we can thus find---for $j$ sufficiently large---graphs $u_{j}:\Sigma_{\infty}\to\RR$ so that the exponential normal graph of $u_{j}$ is some component of $\Sigma_{j}$. We have that $u_{j}\to 0$ in $C^{\infty}(\Sigma_{\infty})$. After normalizing by $\Vert u_{j} \Vert_{C^{2,\alpha}(\Sigma_{\infty})}$, it is standard to pass to a further subsequence to find a non-trivial Jacobi field on (a cover of) $\Sigma_{\infty}$. This contradicts the bumpyness of $(M^{n},g)$. 
		
		Alternatively, the implicit function theorem guarantees that any strictly stable minimal surface has a (strictly) mean convex tubular neighborhood (as before, passing to a double cover if the surface is one-sided). No minimal surface (besides the original one) can be completely contained in this neighborhood by the maximum principle. However, if there were infinitely many embedded stable minimal surfaces with bounded area, then curvature estimates and area bounds would allow one to find a sequence of such surfaces converging (possibly with multiplicity) to a limiting such surface. This cannot occur by considering the aforementioned mean convex neighborhood. 
	\end{proof}
	
	\begin{lemm}\label{lemm:stab-count-unstab-count}
	For $3\leq n \leq 7$, consider $(M^{n},g)$ a connected closed Riemannian $n$-manifold. Assume there are at most $N$ stable minimal surfaces in $(M^{n},g)$. Then $M$ contains at most $\max\{2N,1\}$ disjoint closed embedded unstable minimal hypersurfaces.
	\end{lemm}
	\begin{proof}[Proof of Lemma \ref{lemm:stab-count-unstab-count}]
		Suppose there are $Q\geq2$ disjoint unstable embedded minimal hypersurfaces $\Sigma_{1},\dots,\Sigma_{Q}$. Consider 
		\[
		\tilde M : = M \setminus \bigcup_{i=1}^{Q}\Sigma_{i}.
		\]
		We can give $\tilde M$ the structure of a compact Riemannian manifold with boundary $\partial \tilde M$ equal to the $\Sigma_{i}$ (or their two-sided double cover if they were one-sided originally). Note that each component of $\partial\tilde M$ is unstable (for the components that were originally two-sided, this is clear; for the one-sided surfaces this follows by lifting an unstable variation to the double cover). Write $\tilde M = \tilde M_{1} \cup \tilde M_{2}$, where $\tilde M_{1}$ is the union of the connected components of $\tilde M$ with exactly one boundary component and $\tilde M_{2}$ is the union of components of $\tilde M$ with at least $2$ boundary components. 
		
		We claim that $\partial\tilde M_{2}$ has at least $Q$ components. Indeed, consider some $\Sigma_{i}$. Suppose first that $\Sigma_{i}$ is two sided and separating. Consider $\Omega_{\pm}$ the components of $\tilde M$ containing $\Sigma_{i}$ in their boundary. Suppose that $\Omega_{-}\cup\Omega_{+}\subset \tilde M_{1}$. In this case, it is clear that $Q=1$ and $\Sigma_{i}$ was the only unstable surface in the family, contradicting our assumption that $Q\geq 2$. Thus, at least one of the $\Omega_{\pm}$ is contained in $\tilde M_{2}$. This determines at least one element of $\partial \tilde M_{2}$. Now, suppose that $\Sigma_{i}$ is two-sided but not separating. It is clear that any component of $\partial \tilde M$ associated to $\Sigma_{i}$ cannot bound a component of $\tilde M_{1}$ (otherwise $\Sigma_{i}$ would be separating). Finally, if $\Sigma_{i}$ was one-sided and some component of $\partial \tilde M$ bounded a component of $\tilde M_{1}$, then as before, we would have $Q=1$, a contradiction.

		Given this, Lemma \ref{lemm:stab-count-unstab-count} is a consequence of the following claim:

		\begin{clai}
			For $3\leq n \leq 7$, consider $(\check M^{n},g)$ a connected compact Riemannian manifold with $J\geq 2$ boundary components that are all unstable minimal hypersurfaces. Then, there are at least $J/2$ closed embedded stable minimal surfaces in the interior of $\check M$. 
		\end{clai}
		
		\begin{proof}[Proof of the claim]
			Write the components of $\partial \check M$ as $\Gamma_{1},\dots,\Gamma_{J}$. Because $J\geq 2$, 
			\[
			[\Gamma_{i}]\not = 0 \in H_{n-1}(\check M;\ZZ_{2}).
			\] 
			Find $\Gamma_{1}' \in [\Gamma_{1}]$ minimizing area in the homology class. Because each component of the boundary is an unstable minimal surface, $\Gamma_{1}'$ is contained in the interior of $\check M$. There is thus an open set $\check \Omega_{1}\subset \check M$ so that $\partial\check\Omega_{1} = \Gamma_{1}\cup\Gamma_{1}'$. Set $\check M_{1} : = \check M \setminus \check\Omega_{1}$. This manifold has $J-1$ unstable boundary components and at least $1$ stable boundary component. For $\Gamma_{2}$ one of the unstable boundary components minimize area in homology to get $\Gamma_{2}'$. Each component of $\Gamma_{2}'$ is contained in the interior of $\check M_{1}$ or coincides with one of the stable components of $\partial \check M_{1}$. In either case, we can repeat this process $J$ times. Consider the disjoint stable minimal surfaces $\Gamma_{1}',\dots,\Gamma_{J}'$ in $\check M$. We also note that, by construction, the $\check\Omega_{1},\dots,\check\Omega_{J}$ are all disjoint.

			We would like to estimate the number of connected components of $\cup_{j=1}^{J}\Gamma_{j}'$ (note that it is not necessarily true that the $\Gamma_{j}'$ are pairwise disjoint, but at any point of intersection of $\Gamma_{j}'$ and $\Gamma_{k}'$, the two surfaces locally agree, by construction). Consider a connected component $\Gamma \subset \cup_{j=1}^{J}\Gamma_{j}'$. There is some $a\in\{1,\dots,J\}$ so that $\Gamma\subset \partial\check\Omega_{a}$. Moreover, there is at most one other index $b \in\{1,\dots,J\}\setminus\{a\}$ so that $\Gamma\subset \partial\check\Omega_{b}$ (this holds because the $\check\Omega_{j}$ are pairwise disjoint). Finally, for any $\Gamma_{j}$, we can find such a component $\Gamma$ (possibly multiple such components). Choosing such a component $\Gamma$ corresponding to each $\Gamma_{j}$, we have just seen that a component can appear at most twice in this process. Thus we have found at least $J/2$ components, proving the claim.
		\end{proof}
		This completes the proof of Lemma \ref{lemm:stab-count-unstab-count}. 
	\end{proof}

\begin{proof}[Proof of Theorem \ref{theo:area-unbounded}]
If the first possibility fails, then there is a uniform area bound for embedded stable minimal hypersurfaces. Thus, Lemma \ref{lemm:stab-finitely-many} implies that there are finitely many stable embedded minimal hypersurfaces in $(M^{n},g)$. Call this number $N$. Applying Lemma \ref{lemm:stab-count-unstab-count}, we find that there are at most $N+\max\{2N,1\}$ pairwise disjoint embedded minimal surfaces in $(M,g)$. Applied to $\Sigma_{p}$, there is at least one component $\Sigma_{p}'$ with $\area(\Sigma_{p}') \geq C p^{\frac 1n}$. 
\end{proof}

\bibliographystyle{alpha}
\bibliography{main}

\end{document}